\documentclass[11pt]{amsart}
\usepackage[latin9]{inputenc}
\usepackage{mathrsfs}
\usepackage{mathtools}
\usepackage{amstext}
\usepackage{amsthm}
\usepackage{amssymb}
\usepackage{graphicx}
\usepackage[unicode=true,
 bookmarks=false,
 breaklinks=false,pdfborder={0 0 1},backref=section,colorlinks=false]
 {hyperref}

\makeatletter
\numberwithin{equation}{section}
\numberwithin{figure}{section}
\theoremstyle{plain}
\newtheorem{thm}{\protect\theoremname}[section]
\theoremstyle{remark}
\newtheorem{rem}[thm]{\protect\remarkname}
\theoremstyle{definition}
\newtheorem{defn}[thm]{\protect\definitionname}
\theoremstyle{plain}
\newtheorem{lem}[thm]{\protect\lemmaname}
\theoremstyle{plain}
\newtheorem{fact}[thm]{\protect\factname}

\makeatother

\providecommand{\definitionname}{Definition}
\providecommand{\factname}{Fact}
\providecommand{\lemmaname}{Lemma}
\providecommand{\remarkname}{Remark}
\providecommand{\theoremname}{Theorem}

\begin{document}
\title[Eigenvalue estimates for the poly-Laplace]{Eigenvalue estimates for the poly-Laplace operator on lattice subgraphs}
\author{Bobo Hua and Ruowei Li}
\address{Bobo Hua: School of Mathematical Sciences, LMNS, Fudan University,
Shanghai 200433, China; Shanghai Center for Mathematical Sciences,
Jiangwan Campus, Fudan University, No. 2005 Songhu Road, Shanghai
200438, China.}
\email{bobohua@fudan.edu.cn}
\address{Ruowei Li: School of Mathematical Sciences, Fudan University, Shanghai
200433, China; Shanghai Center for Mathematical Sciences, Jiangwan
Campus, Fudan University, No. 2005 Songhu Road, Shanghai 200438, China;
MPI MiS Leipzig, 04103 Leipzig, Germany.}
\email{rwli19@fudan.edu.cn}
\begin{abstract}
We introduce the discrete poly-Laplace operator on a subgraph with
Dirichlet boundary condition. We obtain upper and lower bounds for
the sum of the first $k$ Dirichlet eigenvalues of the poly-Laplace
operators on a finite subgraph of lattice graph $\mathbb{Z}^{d}$
extending classical results of Li-Yau and Kröger. Moreover, we prove
that the Dirichlet $2l$-order poly-Laplace eigenvalues are at least
as large as the squares of the Dirichlet $l$-order poly-Laplace eigenvalues.
\end{abstract}

\maketitle
\tableofcontents{}

\section{Introduction}

Let $\Omega$ be a connected bounded domain with a smooth boundary
in a $d$-dimensional Euclidean space $\mathbb{R}^{d}$ and let $\nu$
be the outward unit normal vector field of $\partial\Omega$. Denote
by $\Delta$ the Laplace operator on $\mathbb{R}^{d}$. Solutions
of $\Delta f=0$ on $\Omega$ are the classical harmonic functions
that describe the equilibrium state of an elastic homogeneous membrane.
Solutions of $\Delta^{2}f=0$ are called biharmonic, and they model
equilibria of homogeneous plates. Similarly, solutions of $\Delta^{l}f=0$,
where $l\in\mathbb{N}_{+}$, are called poly-harmonic. One naturally
considers the eigenvalue problem
\begin{equation}
\left\{ \begin{aligned} & (-\Delta)^{l}f=\lambda^{l}f\quad\quad\text{in}\ \Omega,\\
 & f|_{\partial\Omega}=\frac{\partial f}{\partial\nu}|_{\partial\Omega}=\ldots=\frac{\partial^{l-1}f}{\partial^{l-1}\nu}|_{\partial\Omega}=0.
\end{aligned}
\right.\label{eq: eigenvalue problem in R^d}
\end{equation}
We can order the eigenvalues as 
\[
0<\lambda_{1}^{l}(\Omega)\leqslant\lambda_{2}^{l}(\Omega)\leqslant\ldots\leqslant\lambda_{k}^{l}(\Omega)\leqslant\ldots.
\]
For the case $l=1$, the eigenvalue problem (\ref{eq: eigenvalue problem in R^d})
is called a fixed membrane problem which has been well-studied. Beginning
with the work of Weyl in 1912 \cite{Weyl1912}, he states that as
$k\rightarrow\infty$, 
\[
\lambda_{k}^{1}(\Omega)\sim C_{d}(\frac{k}{V(\Omega)})^{\frac{2}{d}},
\]
where $V(\Omega)$ is the volume of the domain $\Omega$ and $C_{d}=(2\pi)^{2}V_{d}^{-\frac{2}{d}}$
with $V_{d}$ being the volume of the unit ball in $\mathbb{R}^{d}$. 

For a planar domain $\Omega$ that tiles $\mathbb{R}^{2}$, Pólya
\cite{Polya1961} provided the lower bound estimate for $d=2$ as
\[
\lambda_{k}^{1}(\Omega)\geqslant C_{d}(\frac{k}{V(\Omega)})^{\frac{2}{d}}.
\]
His proof applies to any dimension $d\geqslant2$, and he conjectured
that this estimate holds for all bounded domains in $\mathbb{R}^{d}$.
A partial solution to this conjecture was obtained by Lieb \cite{Lieb1980}
\[
\lambda_{k}^{1}(\Omega)\geqslant D_{d}(\frac{k}{V(\Omega)})^{\frac{2}{d}},
\]
where the constant $D_{d}<C_{d}$ and is proportional to $C_{d}$.
Later Li-Yau \cite{Li1983} and Berezin \cite{Berezin1972} derived
the sharp estimate known as the Berezin-Li-Yau inequality
\[
\frac{1}{k}\sum_{i=1}^{k}\lambda_{i}^{1}(\Omega)\geqslant\frac{d}{d+2}C_{d}(\frac{k}{V(\Omega)})^{\frac{2}{d}}.
\]
Then Li and Yau\textquoteright s estimate on the sum of the eigenvalues
was improved by Melas \cite{Melas2003} whose result was further improved
in \cite{Kovarik2009,Yolcu2012}. The upper bound estimate on the
sum of the eigenvalues was proved by Kröger \cite{Kroger1994} in
1994, which includes an extra term depending on the geometry of $\Omega$.
Moreover in the same spirit, Kröger \cite{Kroger1992,Kroger1994}
obtained upper and lower bounds for the sum of the first $k$ Neumann
eigenvalues. Finally, Strichartz \cite{Strichartz1996} developed
a general framework for obtaining upper and lower bounds for the sum
of the first $k$ Dirichlet and Neumann eigenvalues. 

For $l\geqslant2$, poly-Laplace operators have various applications
in physics. For example, for $l=2$ the eigenvalue problem (\ref{eq: eigenvalue problem in R^d})
is called a clamped plate problem \cite[Chapter 11]{Henrot2006}.
Agmon \cite{Agmon1965} and Pleijel \cite{Pleijel1950} established
the Weyl\textquoteright s asymptotics for the eigenvalues of the clamped
plate problem. And Birman and Solomyak \cite{Birman1979,Birman1980}
generalized the Weyl\textquoteright s asymptotics to arbitrary $l$.
Later, Levine and Protter \cite{Levine1985} extended the Berezin-Li-Yau
inequality for arbitrary order $l\geqslant1$, 
\begin{equation}
\frac{1}{k}\sum_{i=1}^{k}\lambda_{i}^{l}(\Omega)\geqslant\frac{d}{d+2l}C_{d}^{l}(\frac{k}{V(\Omega)})^{\frac{2l}{d}}+\frac{d}{d+2l}C_{d}^{l}(\frac{k}{V(\Omega)})^{\frac{2(l-1)}{d}},\label{eq: Berezin-Li-Yau inequality for l=00005Cgeq1}
\end{equation}
which is sharp in the sense of Weyl\textquoteright s asymptotics.
Later \cite{Chen2013} gave a significant improvement of (\ref{eq: Berezin-Li-Yau inequality for l=00005Cgeq1})
by adding $l$ lower-order terms than $k^{2l/d}$ to its right-hand
side, and it was further improved by \cite{Wei2016,Yildirim2014}.
For the upper bound of the sum of the first $k$-th eigenvalues of
the Dirichlet bi-Laplace operator, Cheng and Wei in \cite{Cheng2013}
prove that there exists a constant $r_{0}>0$ such that for $k\geqslant V(\Omega)r_{0}^{n}$,
\begin{equation}
\frac{1}{k}\sum_{j=1}^{k}\lambda_{j}^{2}(\Omega)\leqslant\frac{1+\frac{4(d+4)(d^{2}+2d+6)}{d+2}\frac{V(\Omega_{r_{0}})}{V(\Omega)}}{\left(1-\frac{V(\Omega_{r_{0}})}{V(\Omega)}\right)^{\frac{d+4}{d}}}\frac{dC_{d}^{2}}{d+4}(\frac{k}{V(\Omega)})^{\frac{4}{d}}.\label{eq: upper bound in continuous case}
\end{equation}
where $\Omega_{r}\coloneqq\left\{ x\in\Omega:\text{dist}(x,\partial\Omega)<r^{-1}\right\} $.
The estimate is sharp in the sense of the Weyl\textquoteright s asymptotics
since $V(\Omega_{r_{0}})\to0$ as $r_{0}\to+\infty$. The eigenvalue
problem (\ref{eq: eigenvalue problem in R^d}) for arbitrary order
case has recently received more attention. For instance, look at \cite{Jost2011}
for universal bounds for eigenvalues, \cite{Forster2008,Laptev2000,Netrusov1996}
for Lieb-Thirring type inequalities and \cite{Gazzola1991} for positiving
preserving problems.

In recent years, increasing attention has been paid to the analysis
of graphs. The graph Laplacian has been extensively studied in the
literature, e.g. \cite{Grone1990,Guattery2000,Hua2014,Bauer2022}.
Hirschler and Woess in \cite{Hirschler2021} studied bi-Laplace equations
for directed networks and Markov chains. And Duffin considered the
difference equations of the poly-Laplace type \cite{Duffin1958}.
Other works \cite{Feng2012,Schweiger2021,Khalkhuzhaev2019} mainly
focus on the numerical framework of bi-Laplace. However, there are
no references available studies on the eigenvalue problem of the discrete
poly-Laplace operators. The motivation comes from studying the discretization
of the corresponding physical models and providing theoretical results
for future numerical calculations and algorithms. In this paper, we
first introduce the definition of the poly-Laplace operator on a locally
finite connected graph $G$ with Dirichlet boundary condition, and
moreover, we prove the upper and lower estimates for the sum of the
first $k$ Dirichlet eigenvalues of poly-Laplace operators on a lattice
graph.

Recall the setting of graphs. A locally finite, simple, undirected
and connected graph $G=(V,E)$ consists of the set of vertices $V$
and the set of edges $E$. In the following $x\in G$ means that $x\in V$,
and let $|G|$ be the number of vertices in $G$. Two vertices $x,y$
are called neighbors, denoted by $x\sim y$, if there exists $e\in E$
connecting $x$ and $y$. The degree of $x\in V$ is defined as $\text{deg}(x)\coloneqq|\left\{ y:y\sim x\right\} |.$
The graph distance on $G$ is defined by 
\[
d(x,y):=\inf\{n|x=z_{0}\sim\dots\sim z_{n}=y\}.
\]
Define balls $B(x,r)=\{y\in V:d(x,y)\leqslant r\}$ on $G$. A subgraph
$G_{1}=(V_{1},E_{1})$ of $G$ means $V_{1}\subset V$ and $E_{1}\subset E$.
For a subgraph $\Omega$ of $G$, recall the notion of the boundary
$\delta\Omega$
\[
\delta\Omega\coloneqq\left\{ y\in V\backslash\Omega:\exists x\in\Omega,\text{s.t. }y\sim x\right\} .
\]
Denote $\delta\Omega$ by $\delta_{1}\Omega$, and for a positive
integer $l$ define the $l$-th layer boundary as
\[
\delta_{l}\Omega\coloneqq\delta(\Omega\cup\delta_{1}\Omega\cup\ldots\cup\delta_{l-1}\Omega).
\]
Recall that the graph Laplace operator is defined for any $f:V\rightarrow\mathbb{C}$
as
\[
\Delta f(x)=\sum_{y\sim x}(f(y)-f(x)).
\]
For a positive integer $l$, the poly-Laplace operator can be regarded
as iterations of the Laplace operator defined inductively 
\[
\Delta^{l}=\underset{l-\text{th}}{\underbrace{\Delta\circ\Delta\circ\cdots\circ\Delta}}.
\]
For a finite graph $G=(V,E)$ with $V=\left\{ x_{1},\cdots,x_{n}\right\} $,
the Laplace matrix is defined as $\Delta=A-D$, where $D=\text{diag}(\text{deg}(x_{1}),\cdots,\text{deg}(x_{n}))$
is the degree matrix and $A$ is the adjacency matrix whose element
$A_{ij}$ is one when there is an edge from vertex $x_{i}$ to vertex
$x_{j}$, and zero when there is no edge. For any $f:V\rightarrow\mathbb{C}$,
\begin{align*}
(-\Delta)^{l}f(x)=(D-A)^{l}f(x) & =\sum_{m=0}^{l}\left(\begin{array}{c}
l\\
m
\end{array}\right)D^{l-m}(-1)^{m}A^{m}f(x)\\
 & \eqqcolon\underset{y\in B(x,l)}{\sum}a_{xy}^{l}f(y),
\end{align*}
where 
\begin{equation}
a_{xy}^{l}\coloneqq\sum_{m=0}^{l}\left(\begin{array}{c}
l\\
m
\end{array}\right)(-1)^{m}(\text{deg}(y))^{l-m}p_{m}(x,y)\label{eq:a_xy}
\end{equation}
and $p_{m}(x,y)=\sharp\left\{ \text{paths }x=z_{0}\sim\dots\sim z_{m}=y\right\} $
is the number of paths of length $m$ from vertex $x$ to vertex $y$. 

Given a finite subgraph $\Omega\subseteq G$, we introduce the poly-Laplace
operator with Dirichlet boundary, denoted by $\Delta_{\Omega}^{l,\mathcal{D}}$,
which is defined as
\[
\Delta_{\Omega}^{l,\mathcal{D}}f\coloneqq\Delta^{l}f^{*}|_{\Omega}.
\]
Here $f:\Omega\rightarrow\mathbb{C}$, and $f^{*}:G\rightarrow\mathbb{C}$
is the zero extension of $u$. One readily sees that $(-1)^{l}\Delta_{\Omega}^{l,\mathcal{D}}$
is a self-adjoint and positive definite operator on $\Omega$. The
Dirichlet poly-Laplace eigenvalue problem is given by 
\begin{equation}
(-1)^{l}\Delta_{\Omega}^{l,\mathcal{D}}f(x)=\lambda^{l}f(x)\;\text{in \ensuremath{\Omega}}.\label{eq: main poly-laplace Dirichlet eigenvalue problem}
\end{equation}
There are $|\Omega|$ corresponding real eigenvalues: 
\[
0<\lambda_{1}^{l}(\Omega)\leqslant\lambda_{2}^{l}(\Omega)\leqslant\ldots\leqslant\lambda_{|\Omega|}^{l}(\Omega).
\]
See subsection \ref{subsec: Dirichlet poly-Laplace} for more properties
of this operator. 

Consider the $d$-dimensional integer lattice $\mathbb{Z}^{d}$ as
a graph. We study the Dirichlet poly-Laplace problem on a finite subgraph
$\Omega$ of $\mathbb{Z}^{d}$ using the Fourier transform. In 2022,
Bauer and Lippner \cite{Bauer2022} first studied the graph Laplacian
eigenvalues with Dirichlet boundary condition on $\Omega$ of $\mathbb{Z}^{d}$.
Later Wang \cite{Wang2023} introduced the Dirichlet problems of graph
fractional Laplacian and developed a framework for estimating the
Dirichlet eigenvalues of fractional Laplacian. For the average of
the first $k$ Dirichlet eigenvalues, they derived the upper estimate
in the spirit of Kröger \cite{Kroger1994} and the lower estimate
in the spirit of Li-Yau \cite{Li1983}. In this paper, we follow the
proof strategies in \cite{Kroger1994,Li1983,Bauer2022,Wang2023}to
derive the upper and lower bounds for the Dirichlet eigenvalues of
the discrete poly-Laplace operator. 

The following are the main results of this paper.
\begin{thm}
\label{thm: main thm}For a positive integer $l$, consider the Dirichlet
poly-Laplace eigenvalue problem (\ref{eq: main poly-laplace Dirichlet eigenvalue problem})
on a finite subgraph $\Omega$ of $\mathbb{Z}^{d}$. Then

$(a)$ for all $1\leqslant k\leqslant\min\{1,\frac{V_{d}}{2^{d}}\}|\Omega|$,
we have the upper bound estimate 
\[
\frac{1}{k}\sum_{j=1}^{k}\lambda_{j}^{l}(\Omega)\leqslant(2\pi)^{2l}\frac{d}{d+2l}(\frac{k}{V_{d}|\Omega|})^{\frac{2l}{d}}+\frac{|\partial^{l}\Omega|}{|\Omega|}.
\]
For $1\leqslant k\leqslant\min\{1,\frac{V_{d}}{2^{d+1}}\}|\Omega|$,
there holds 
\[
\lambda_{k+1}^{l}(\Omega)\leqslant(2\pi)^{2l}\frac{d\cdot2^{\frac{d+2l}{d}}}{d+2l}(\frac{k}{V_{d}|\Omega|})^{\frac{2l}{d}}+2\frac{|\partial^{l}\Omega|}{|\Omega|}.
\]
Recall that $V_{d}$ is the volume of the unit ball in $\mathbb{R}^{d}$
and 
\begin{align*}
|\partial^{l}\Omega| & =\underset{x\in\delta\Omega\cup\ldots\cup\delta_{l}\Omega\,}{\sum}\underset{y\in B(x,l)\cap\Omega}{\sum}|a_{xy}^{l}|\\
 & \leqslant4^{l}d^{l}\left(|\delta\Omega|+\ldots+|\delta_{l}\Omega|\right),
\end{align*}
where $a_{xy}^{l}$ is defined by (\ref{eq:a_xy}).

$(b)$ For all $1\leqslant k\leqslant\min\{1,(\frac{\sqrt{6}}{2\pi})^{d}V_{d}\}|\Omega|$,
the lower bound estimate is given as 
\begin{align*}
\lambda_{k}^{l}(\Omega) & \geqslant\frac{1}{k}\sum_{j=1}^{k}\lambda_{j}^{l}(\Omega)\\
 & \geqslant\sum_{m=0}^{l}\left(\begin{array}{c}
l\\
m
\end{array}\right)\left(-\frac{1}{12}\right)^{m}(2\pi)^{2(l+m)}\frac{d}{d+2(l+m)}(\frac{k}{V_{d}|\Omega|})^{\frac{2(l+m)}{d}}>0.
\end{align*}
\end{thm}

\begin{rem}
(1) Compared with the continuous upper bound estimates (\ref{eq: upper bound in continuous case})
and lower bound estimates (\ref{eq: Berezin-Li-Yau inequality for l=00005Cgeq1}),
the terms $(2\pi)^{2l}\frac{d}{d+2l}(\frac{k}{V_{d}|\Omega|})^{\frac{2l}{d}}$
are the same and consistent with the Weyl\textquoteright s asymptotics.
Moreover, the upper bound estimates (\ref{eq: upper bound in continuous case})
also contain a geometric quantity $V(\Omega_{r_{0}})$ that depends
on the domain boundary. 

(2) The boundary term $|\partial^{l}\Omega|$ can be bounded by $4^{l}d^{l}\left(|\delta\Omega|+\ldots+|\delta_{l}\Omega|\right)$.
In particular, for $l=1$ and $l=2$, we can bound $|\partial^{1}\Omega|$
and $|\partial^{2}\Omega|$ by (\ref{eq: |=00005Cpartial^=00007B1=00007D=00005COmega|})
and (\ref{eq: |=00005Cpartial^=00007B2=00007D=00005COmega|}) respectively.
\end{rem}

The ideas behind the proof are as follows. The upper bound estimate
is based on Lemma \ref{lem: upper bound lemma} (was proved by the
Rayleigh-Ritz formula in \cite{Bauer2022}), which gives an estimate
(\ref{eq: upper bound lemma of L}) for eigenvalues of general self-adjoint
and positive semidefinite operators. Let $g(x)=h_{z}(x)=e^{i\langle x,z\rangle}$
and integrate (\ref{eq: upper bound lemma of L}) in a measurable
subset $B$ of $[-\pi,\pi]^{d}$. By the properties of the Fourier
transform we can prove Lemma~\ref{lem: upper bound lemma of poly-Laplace}.
Finally, the upper bound estimate $(a)$ follows from the proper choice
for set $B$.

For the lower bound estimate, the key is to prove Lemma~\ref{lem: lower bound lemma}
which is an adaption of Li and Yau's method \cite{Li1983}. To prove
the lemma, for a real-valued function $F$ with $0\leqslant F\leqslant M$
and $\int_{[-\pi,\pi]^{d}}F(z)dz\geqslant K,$ first construct a radially
symmetric function $\varphi(z):[-\pi,\pi]^{d}\rightarrow\mathbb{C}$
which satisfies $0\leqslant\varphi(z)\leqslant(\Phi(z))^{l}$ and
\[
\int_{[-\pi,\pi]^{d}}(\Phi(z))^{l}F(z)dz\geqslant\int_{[-\pi,\pi]^{d}}\varphi(z)F(z)dz.
\]
Note that $\widetilde{F}=M1_{B_{R}}$ minimizes the integral $\int_{[-\pi,\pi]^{d}}\varphi(z)F(z)dz$,
and gives a lower bound of $\int_{[-\pi,\pi]^{d}}(\Phi(z))^{l}F(z)dz$.
Choose $F(z)=\sum_{j=1}^{k}|\langle\phi_{j}^{l},h_{z}\rangle_{\Omega}|^{2}$,
where $\{\phi_{i}^{l}\}_{1\leqslant i\leqslant|\Omega|}$ are corresponding
eigenfunctions of $(-1)^{l}\Delta_{\Omega}^{l,\mathcal{D}}$, and
$(b)$ is derived from Lemma~\ref{lem: lower bound lemma}. 

Notably, on a graph $G$ without the boundary, the eigenvalues of
$(-\Delta)^{l}$ are actually the $l$-th power of the corresponding
eigenvalues of the Laplace operator $-\Delta$. However, this relationship
generally does not hold for the Dirichlet eigenvalue problems. The
following theorem shows that the Dirichlet eigenvalues of lower-order
poly-Laplace can provide a rough lower bound for the higher-order
poly-Laplace. Therefore, it is necessary to establish a more precise
estimate for the poly-Laplace eigenvalue, as stated in Theorem \ref{thm: main thm}.
\begin{thm}
\label{thm:bi-laplace and laplace}For a finite subgraph $\Omega$
of $G=(V,E)$ and a positive integer $l$, let $\lambda_{k}^{l}$,
$\lambda_{k}^{2l}$ be the $k$-th eigenvalues of the Dirichlet poly-Laplace
$\Delta_{\Omega}^{l,\mathcal{D}}$ and $\Delta_{\Omega}^{2l,\mathcal{D}}$
respectively. Then for $1\leqslant k\leqslant|\Omega|,$
\[
\left(\lambda_{k}^{l}\right)^{2}\leqslant\lambda_{k}^{2l}.
\]
Moreover, if $\Omega$ is a finite subgraph of $\mathbb{Z}^{d}$,
then 
\[
\left(\lambda_{k}^{l}\right)^{2}<\lambda_{k}^{2l}.
\]
\end{thm}

\begin{rem}
\label{rem: bilaplace estimate}(1) We obtain the same estimate $\left(\lambda_{k}^{l}\right)^{2}\leqslant\lambda_{k}^{2l}$
for the continuous settings, see Theorem \ref{thm: continuous laplace and bi-laplace}. 

(2) The strict inequality follows from Lemma \ref{lem:eigenfunction on Z^d},
which states that there is no $\ell^{2}$-eigenfunction of poly-Laplace
on $\mathbb{Z}^{d}$, while the strictness is usually not true on
general graphs. For example, on the subgraph $\Omega=\left\{ v_{1},v_{2}\right\} $
of the $3$-cycle $C_{3}=\left\{ v_{1},v_{2},v_{3}\right\} $, we
have $\left(\lambda_{2}^{1}\right)^{2}=(-3)^{2}=9=\lambda_{2}^{2}$.

(3) On a path graph $[0,n]\subseteq\mathbb{Z}$, we have $\frac{\left(\lambda_{k}^{1}\right)^{2}}{\lambda_{k}^{2}}\to c_{k}<1$
as $n\to+\infty$, see numerical experiments in Appendix. 
\end{rem}

For an infinite graph $G=(V,E)$, the $k$-th eigenvalue of the poly-Laplace
on $G$ is defined as 
\[
\lambda_{k}^{l}(G)=\underset{\begin{array}{c}
M\subseteq C_{0}(V)\\
\text{dim}M=k
\end{array}}{\text{inf}}\underset{\begin{array}{c}
f\in M\\
\|f\|_{2}=1
\end{array}}{\text{max}}\langle(-\Delta)^{l}f,f\rangle_{G}.
\]
And it can be estimated by the exhaustion trick, see Definition \ref{def:exhaustion}.
\begin{thm}
\label{thm:exhaustion}Let $\mathcal{W}=\left\{ {W_{i}}\right\} _{i=1}^{\infty}$
be an exhaustion of infinite graph $G=(V,E)$ with bounded degree
$\underset{x\in V}{\text{sup}}\text{ deg}x<+\infty$, then 
\[
\lambda_{k}^{l}(G)=\underset{i\to+\infty}{lim}\lambda_{k}^{l}(W_{i}).
\]
In particular, for $k=1$ we have 
\[
\left(\lambda_{1}^{1}(G)\right)^{l}=\lambda_{1}^{l}(G)=\underset{i\to+\infty}{lim}\lambda_{1}^{l}(W_{i}).
\]
\end{thm}

\begin{rem}
For $k=1$, the bottom of the spectrum of poly-Laplace $(-\Delta)^{l}$
is the $l$-th power of the bottom of the spectrum of Laplace operator
$-\Delta$.
\end{rem}

The paper is organized as follows. In next section, we first recall
the basic setting of graphs, and introduce the discrete poly-Laplace
operator and some important lemmas. In Section \ref{sec:Proof-of-Theorem},
we prove Theorem \ref{thm: main thm} and Theorem \ref{thm:bi-laplace and laplace}. 

\section{Preliminaries}

\subsection{Basic notations and facts of graphs}

We recall the setting of graphs. A locally finite, simple, undirected
and connected graph $G=(V,E)$ consists of the set of vertices $V$
and the set of edges $E$. We denote the space of functions on $G$
by $C(G):=\{f:V\rightarrow\mathbb{C}\}$, and let $C_{0}(G)$ be the
set of all functions on $G$ with finite support. For a finite graph
$G=(V,E)$, then $C(G)$ is a finite dimensional complex vector space
of functions defined on $G$, and equip it with the natural Hermitian
scalar product 
\[
\left\langle f,g\right\rangle _{G}=\underset{x\in G}{\sum}f(x)g(x),\,\forall f,g\in C(G).
\]
For $f\in C(G)$ and $p\geqslant1$, define the $\ell^{p}$-norm of
$f$ as 
\[
\|f\|_{\ell^{p}(G)}=\left(\underset{x\in G}{\sum}|f(x)|^{p}\right)^{\frac{1}{p}}.
\]
And let
\[
\ell^{p}(G)\coloneqq\left\{ f\in C(G):\|f\|_{\ell^{p}(G)}<+\infty\right\} 
\]
be the space of $\ell^{p}$-summable functions on $G$. For any edge
$(x,y)\in E$ define the gradient operator as
\[
\nabla_{xy}f=f(y)-f(x),\,\forall f\in C(G).
\]
The Laplace operator on $G$ is defined as 
\[
\Delta_{G}f(x)=\sum_{y\sim x}\nabla_{xy}f,\,\forall f\in C(G),
\]
which is self-adjoint with respect to the above scalar product. And
we write $\Delta$ instead of $\Delta_{G}$ when it causes no confusion.

For subgraphs $G_{1},G_{2}\subseteq G$, define the edge set of connecting
$G_{1}$ and $G_{2}$ as 
\[
E(G_{1},G_{2})\coloneqq\left\{ (x,y)\in E:x\in G_{1},y\in G_{2}\right\} .
\]
Let $\Omega$ be a finite subgraph of $G$. Give the notions of the
boundary $\delta{\Omega}$ and the closure $\bar{\Omega}$ of the
subgraph $\Omega$ as

\[
\begin{array}{c}
\delta\Omega\coloneqq\left\{ y\in V\backslash\Omega:\exists x\in\Omega,\text{s.t. }y\sim x\right\} ,\\
\bar{\Omega}\coloneqq\Omega\cup\delta\Omega.
\end{array}
\]
Denote $\delta\Omega$ by $\delta_{1}\Omega$, and for a positive
integer $l$ define the $l$-th layer boundary as
\[
\delta_{l}\Omega\coloneqq\delta(\Omega\cup\delta_{1}\Omega\cup\ldots\cup\delta_{l-1}\Omega).
\]
For $f\in C(\Omega)$, the Dirichlet Laplace operator is defined as
\[
\Delta_{\Omega}^{\mathcal{D}}f\coloneqq\Delta f^{*}|_{\Omega},
\]
where $f^{*}\in C(G)$ obtained by extending functions to be $0$
on $\Omega^{c}$. Note that $\Delta_{\Omega}^{\mathcal{D}}$ is self-adjoint
with respect to the Hermitian scalar product \cite[Proposition 2.2]{Bauer2022}.
For $f\in C(\bar{\Omega})$ and $x\in\delta\Omega$, define the discrete
outward normal derivative of $f$ at $x$ as
\[
\frac{\partial f}{\partial n}(x)\coloneqq\underset{x\sim y\in\Omega}{\sum}\nabla_{yx}f.
\]

\subsection{\label{subsec: Dirichlet poly-Laplace}The Dirichlet poly-Laplace
operator.}

In this subsection, we introduce the Dirichlet poly-Laplace operator
and give the Dirichlet poly-Laplace eigenvalue problem.
\begin{defn}
Let $\Omega$ be a finite subgraph of $G=(V,E)$, then for all functions
$f\in C(\Omega)$ define the Dirichlet poly-Laplace operator $\Delta_{\Omega}^{l,\mathcal{D}}:C(\Omega)\to C(\Omega)$
as 
\[
\Delta_{\Omega}^{l,\mathcal{D}}f\coloneqq\Delta^{l}f^{*}|_{\Omega},
\]
where $f^{*}\in C(G)$ obtained by extending functions to be $0$
on $\Omega^{c}$. In particular, we know $\Delta_{\Omega}^{1,\mathcal{D}}$
is the Dirichlet Laplace operator $\Delta_{\Omega}^{\mathcal{D}}$.
\end{defn}

\begin{lem}
\label{lem: bi-Laplace self-adjoint}Let $\Omega$ be a finite subgraph
of $G=(V,E)$, then $(-1)^{l}\Delta_{\Omega}^{l,\mathcal{D}}$ is
self-adjoint and positive definite.
\end{lem}

\begin{proof}
For all functions $f,g\in C(\Omega)$, we have 
\[
\left\langle (-1)^{l}\Delta_{\Omega}^{l,\mathcal{D}}f,g\right\rangle _{\Omega}=\left\langle \left(-\Delta_{G}\right)^{l}f^{*}|_{\Omega},g\right\rangle _{\Omega}=\left\langle \left(-\Delta_{G}\right)^{l}f^{*},g^{\ast}\right\rangle _{V}.
\]
Since $\Delta_{G}$ is self-adjoint, then 
\[
\left\langle \left(-\Delta_{G}\right)^{l}f^{*},g^{\ast}\right\rangle _{V}=\left\langle f^{*},\left(-\Delta_{G}\right)^{l}g^{\ast}\right\rangle _{V}=\left\langle f,(-1)^{l}\Delta_{\Omega}^{l,\mathcal{D}}g\right\rangle _{\Omega}.
\]
And the positive definiteness follows from the positive definiteness
of $-\Delta_{G}$.
\end{proof}
In this paper, we consider the Dirichlet poly-Laplace eigenvalue problem
on a finite subgraph $\Omega$
\[
(-1)^{l}\Delta_{\Omega}^{l,\mathcal{D}}f=\lambda^{l}f,\;\text{in \ensuremath{\Omega}}.
\]
Since $(-1)^{l}\Delta_{\Omega}^{l,\mathcal{D}}$ is self-adjoint,
positive definite and a finite dimensional operator, then there exist
$|\Omega|$ eigenvalues (with multiplicities), all of which are positive
real numbers. We label them in increasing order 
\[
0<\lambda_{1}^{l}(\Omega)\leqslant\lambda_{2}^{l}(\Omega)\leqslant\ldots\leqslant\lambda_{|\Omega|}^{l}(\Omega).
\]
Denote the corresponding eigenfunctions by $\{\phi_{i}^{l}\}_{1\leqslant i\leqslant|\Omega|}$
respectively. 
\begin{lem}
\textup{(Domain monotonicity)} For finite subgraphs $\Omega_{1}$
and $\Omega_{2}$ of $G=(V,E)$, if $\Omega_{1}\subseteq\Omega_{2}$,
then for all $1\leqslant k\leqslant|\Omega_{1}|$
\[
\lambda_{k}^{l}(\Omega_{1})\geqslant\lambda_{k}^{l}(\Omega_{2}).
\]
\end{lem}

\begin{proof}
Since the poly-Laplace $(-1)^{l}\Delta_{\Omega}^{l,\mathcal{D}}$
can be written in the min-max form 
\begin{equation}
\lambda_{k}^{l}(\Omega)=\underset{\begin{array}{c}
M\subseteq C(\Omega)\\
\text{dim}M=k
\end{array}}{\text{min}}\underset{\begin{array}{c}
f\in M\\
\|f\|_{\ell^{2}(\Omega)}=1
\end{array}}{\text{max}}\langle(-1)^{l}\Delta_{\Omega}^{l,\mathcal{D}}f,f\rangle_{\Omega}.\label{eq: k-th eigenvalue of the poly-Laplace in min-max form}
\end{equation}
\end{proof}
\begin{defn}
For an infinite-dimensional Hilbert space $\mathcal{H}$, the spectrum
of a bounded linear operator $T:\mathcal{H}\to\mathcal{H}$ is defined
as 
\[
\sigma(T)\coloneqq\left\{ \lambda\in\mathbb{C}:\lambda I-T\text{ not invertible}\right\} .
\]
And $\text{inf }\sigma(T)$ is the bottom of the spectrum of $T$.
\end{defn}

For an infinite graph $G=(V,E)$, the $k$-th eigenvalue of the poly-Laplace
on $G$ is defined as
\begin{equation}
\lambda_{k}^{l}(G)=\underset{\begin{array}{c}
M\subseteq C_{0}(V)\\
\text{dim}M=k
\end{array}}{\text{inf}}\underset{\begin{array}{c}
f\in M\\
\|f\|_{\ell^{2}(G)}=1
\end{array}}{\text{max}}\langle(-\Delta)^{l}f,f\rangle_{G}.\label{eq: k-th eigenvalue of the poly-Laplace on G}
\end{equation}

\begin{fact}
On an infinite graph $G=(V,E)$ with bounded degree $\underset{x\in V}{\text{sup}}\text{ deg}x<+\infty$,
the poly-Laplace $(-\Delta)^{l}$ is a bounded linear operator from
$\ell^{2}(G)$ to $\ell^{2}(G)$, then

(1) $\lambda_{k}^{l}(G)\subseteq\sigma((-\Delta)^{l})$,

(2) $\sigma((-\Delta)^{l})=\left(\sigma(-\Delta)\right)^{l}$, 

(3) $\lambda_{1}^{l}(G)=\text{inf }\sigma((-\Delta)^{l})=\left(\text{inf }\sigma(-\Delta)\right)^{l}=\left(\lambda_{1}^{1}(G)\right)^{l}$. 
\end{fact}

\begin{defn}
\label{def:exhaustion}Let $G=(V,E)$ be an infinite graph. A sequence
of subsets of vertices $\mathcal{W}=\left\{ {W_{i}}\right\} _{i=1}^{\infty}$
is called an exhaustion of $G$, if it satisfies 

(1) $W_{1}\subseteq W_{2}\subseteq\cdot\cdot\cdot\subseteq W_{i}\subseteq\cdot\cdot\cdot\subseteq V$, 

(2) $|W_{i}|<\infty$, for all $i=1,2,\cdot\cdot\cdot$,

(3) $V=\cup_{i=1}^{\infty}W_{i}$ .
\end{defn}

By the exhaustion method we can estimate the eigenvalues of the poly-Laplace
on $G$ as Theorem \ref{thm:exhaustion}.
\begin{proof}[Proof of Theorem \ref{thm:exhaustion}]
By the domain monotonicity $\lambda_{k}^{l}(W_{i})\geqslant\lambda_{k}^{l}(W_{i+1})$,
then the limit exists follows from the definition (\ref{eq: k-th eigenvalue of the poly-Laplace in min-max form})
and (\ref{eq: k-th eigenvalue of the poly-Laplace on G}).
\end{proof}

\subsection{The integer lattice $\mathbb{Z}^{d}$ and the Fourier transform on
$\mathbb{Z}^{d}$}

Let $\mathbb{Z}^{d}$ be the $d$-dimensional integer lattice, and
choose $S=\{e_{1},\ldots,e_{d},-e_{1},\ldots,-e_{d}\}$ as the generating
set. Here $e_{i}\in\mathbb{Z}^{d}$ is the vector whose $i$-th component
is $1$ and the rest are $0$. Thus consider $\mathbb{Z}^{d}$ as
the Cayley graph generated by $S$. For any $s\in S$, the gradient
operator $\nabla s:$ $C(\mathbb{Z}^{d})\to C(\mathbb{Z}^{d})$ can
be written for a function $f\in C(\mathbb{Z}^{d})$ as
\[
\nabla sf(x)=f(x+s)-f(x).
\]
Note that the gradient operator on $\mathbb{Z}^{d}$ satisfies the
integration by parts and commutative law as follows.
\begin{fact}
\label{fact: integration by parts and commutative law}\textup{(Integration
by parts)} For all $1\leq i\leq d$ and $f,g\in C_{0}(\mathbb{Z}^{d})$,
we have 
\[
\langle\nabla_{e_{i}}f,g\rangle_{\mathbb{Z}^{d}}=\langle f,\nabla_{-e_{i}}g\rangle_{\mathbb{Z}^{d}}.
\]

\textup{(Commutative law)} For all $1\leq i,j\leq d$ and $f\in C(\mathbb{Z}^{d})$,
we have 
\[
\nabla_{e_{i}}\nabla_{e_{j}}f=\nabla_{e_{j}}\nabla_{e_{i}}f.
\]
\end{fact}

Next, we introduce the Fourier transform on $\mathbb{Z}^{d}$. For
$z\in[-\pi,\pi]^{d}$, we define $h_{z}(x):\mathbb{Z}^{d}\rightarrow\mathbb{C}$
as $h_{z}(x)=e^{i\langle x,z\rangle}$, where $\langle x,z\rangle=\sum_{i=1}^{d}x_{i}z_{i}$.
Thus, for all $f\in\ell^{1}(\mathbb{Z}^{d})$, the Fourier transform
on $\mathbb{Z}^{d}$ is defined by 
\[
\begin{aligned}\mathscr{F}:\quad\mathbb{C}^{\mathbb{Z}^{d}} & \rightarrow\ \mathbb{C}^{[-\pi,\pi]^{d}}\\
f\ \  & \mapsto\ \ \widehat{f},
\end{aligned}
\]
and 
\[
\widehat{f}(z)=\sum_{x\in\mathbb{Z}^{d}}e^{-i\langle x,z\rangle}f(x)=\langle f,h_{z}\rangle_{\mathbb{Z}^{d}},
\]
\[
f(x)=(2\pi)^{d}\int_{[-\pi,\pi]^{d}}e^{i\langle x,z\rangle}\widehat{f}(z)dz.
\]
Let $\Omega$ be a finite subgraph of $\mathbb{Z}^{d}$, and write
$\Phi(z)=\sum_{i=1}^{d}(2-2\cos z_{i})$ for $z\in[-\pi,\pi]^{d}$.
Then we prove several critical lemmas as follows.
\begin{lem}
Let $l$ be a positive integer, then for all $f\in C_{0}(\mathbb{Z}^{d})$
we have 
\[
\widehat{(-\Delta)^{l}f}(z)=(\Phi(z))^{l}\widehat{f}(z),\:\forall z\in[-\pi,\pi]^{d}.
\]
\end{lem}

\begin{proof}
By the definition of Fourier transform, for all $z\in[-\pi,\pi]^{d}$,
we have
\begin{align*}
\widehat{(-\Delta)^{l}f}(z) & =\sum_{x\in\mathbb{Z}^{d}}e^{-i\langle x,z\rangle}(-\Delta)^{l}f(x)\\
 & =\sum_{x\in\mathbb{Z}^{d}}(-\Delta)e^{-i\langle x,z\rangle}(-\Delta)^{l-1}f(x)\\
 & =\Phi(z)\sum_{x\in\mathbb{Z}^{d}}e^{-i\langle x,z\rangle}(-\Delta)^{l-1}f(x)\\
 & =\left(\Phi(z)\right)^{l}\sum_{x\in\mathbb{Z}^{d}}e^{-i\langle x,z\rangle}f(x)=(\Phi(z))^{l}\widehat{f}(z).
\end{align*}
\end{proof}
\begin{lem}
Let $\Omega$ be a finite subgraph of $\mathbb{Z}^{d}$. For all $f\in C(\Omega)$,
the following equalities holds:
\begin{align}
\langle f,f\rangle_{\Omega} & =\frac{1}{(2\pi)^{d}}\int_{[-\pi,\pi]^{d}}|\langle f,h_{z}\rangle_{\Omega}|^{2}dz,\label{eq: Plancherel formula}\\
\langle f,(-1)^{l}\Delta_{\Omega}^{l,\mathcal{D}}f\rangle_{\Omega} & =\frac{1}{(2\pi)^{d}}\int_{[-\pi,\pi]^{d}}(\Phi(z))^{l}|\langle f,h_{z}\rangle_{\Omega}|^{2}dz.\label{eq: bi-Laplace fourier eq}
\end{align}
\end{lem}

\begin{proof}
The result of (\ref{eq: Plancherel formula}) follows from the Plancherel
formula. For (\ref{eq: bi-Laplace fourier eq}), letting $1\leq i_{1}\leq d$,
we compute 
\begin{align*}
 & \frac{1}{(2\pi)^{d}}\int_{[-\pi,\pi]^{d}}(2-2\text{cos}z_{i_{1}})|\langle f,h_{z}\rangle_{\Omega}|^{2}dz\\
= & \frac{1}{(2\pi)^{d}}\int_{[-\pi,\pi]^{d}}|e^{-i\langle e_{i_{1}},z\rangle}-1|^{2}|\langle f,h_{z}\rangle_{\Omega}|^{2}dz\\
= & \frac{1}{(2\pi)^{d}}\int_{[-\pi,\pi]^{d}}|\langle f,(e^{-i\langle e_{i_{1}},z\rangle}-1)h_{z}\rangle_{\Omega}|^{2}dz\\
= & \frac{1}{(2\pi)^{d}}\int_{[-\pi,\pi]^{d}}|\langle f^{\ast},\nabla_{-e_{i_{1}}}h_{z}\rangle_{\mathbb{Z}^{d}}|^{2}dz\\
= & \frac{1}{(2\pi)^{d}}\int_{[-\pi,\pi]^{d}}|\langle\nabla_{e_{i_{1}}}f^{\ast},h_{z}\rangle_{\mathbb{Z}^{d}}|^{2}dz\\
= & \langle\nabla_{e_{i_{1}}}f^{\ast},\nabla_{e_{i_{1}}}f^{\ast}\rangle_{\mathbb{Z}^{d}}=\langle f^{\ast},\nabla_{-e_{i_{1}}}\nabla_{e_{i_{1}}}f^{\ast}\rangle_{\mathbb{Z}^{d}},
\end{align*}
which follows from (\ref{eq: Plancherel formula}) and the integration
by parts, see Fact \ref{fact: integration by parts and commutative law}.
Then by summing over all coordinate directions for $1\leq i_{1},\ldots,i_{l}\leq d$
we get 
\begin{align*}
\frac{1}{(2\pi)^{d}}\int_{[-\pi,\pi]^{d}}(\Phi(z))^{l}|\langle f,h_{z}\rangle_{\Omega}|^{2}dz & =\sum_{i_{1},\ldots,i_{l}=1}^{d}\langle f^{\ast},\nabla_{-e_{i_{l}}}\nabla_{e_{i_{l}}}\cdots\nabla_{-e_{i_{1}}}\nabla_{e_{i_{1}}}f^{\ast}\rangle_{\mathbb{Z}^{d}}\\
 & =\langle f^{\ast},(-\Delta)^{l}f^{\ast}\rangle_{\mathbb{Z}^{d}}=\langle f,(-1)^{l}\Delta_{\Omega}^{l,\mathcal{D}}f\rangle_{\Omega}.
\end{align*}
\end{proof}
\begin{lem}
\label{lem: h_z}Let $\Omega$ be a finite subgraph of $\mathbb{Z}^{d}$,
then the following holds: 

\begin{equation}
\langle h_{z},h_{z}\rangle_{\Omega}=|\Omega|.\label{eq:h_=00007Bz=00007D}
\end{equation}
\begin{equation}
|\langle h_{z},(-1)^{l}\Delta_{\Omega}^{l,\mathcal{D}}\left(h_{z}|_{\Omega}\right)\rangle_{\Omega}|\leqslant(\Phi(z))^{l}|\Omega|+|\partial^{l}\Omega|,\label{eq:=00005CDelta^=00007Bl=00007D}
\end{equation}
where
\begin{align*}
|\partial^{l}\Omega| & =\underset{x\in\delta\Omega\cup\ldots\cup\delta_{l}\Omega\,}{\sum}\underset{y\in B(x,l)\cap\Omega}{\sum}|a_{xy}^{l}|\\
 & \leqslant4^{l}d^{l}\left(|\delta\Omega|+\ldots|\delta_{l}\Omega|\right),
\end{align*}
and $a_{xy}^{l}$ is defined by (\ref{eq:a_xy}). In particular, for
$p=1$ and $p=2$, the terms $|\partial^{1}\Omega|$ and $|\partial^{2}\Omega|$
can be bounded by (\ref{eq: |=00005Cpartial^=00007B1=00007D=00005COmega|})
and (\ref{eq: |=00005Cpartial^=00007B2=00007D=00005COmega|}) respectively.
\end{lem}

\begin{proof}
The first result (\ref{eq:h_=00007Bz=00007D}) follows from 
\[
\langle h_{z},h_{z}\rangle_{\Omega}=\sum_{x\in\Omega}h_{z}(x)\overline{h_{z}}(x)=\sum_{x\in\Omega}1=|\Omega|.
\]
Note that
\[
\begin{aligned}|\langle h_{z},(-1)^{l}\Delta_{\Omega}^{l,\mathcal{D}}\left(h_{z}|_{\Omega}\right)\rangle_{\Omega}|=|\langle h_{z}^{\ast},(-\Delta)^{l}\left(h_{z}|_{\Omega}\right)^{\ast}\rangle_{\mathbb{Z}^{d}}|\end{aligned}
.
\]
Since
\[
\begin{aligned}|\langle h_{z},(-\Delta)^{l}\left(h_{z}|_{\Omega}\right)^{\ast}\rangle_{\mathbb{Z}^{d}}|= & |\widehat{(-\Delta)^{l}\left(h_{z}|_{\Omega}\right)^{\ast}}(z)|\\
= & |(\Phi(z))^{l}\widehat{\left(h_{z}|_{\Omega}\right)^{\ast}}(z)|\\
= & (\Phi(z))^{l}|\Omega|,
\end{aligned}
\]
and by the definition (\ref{eq:a_xy})
\begin{align*}
|\langle h_{z},(-\Delta)^{l}\left(h_{z}|_{\Omega}\right)^{\ast}\rangle_{\Omega^{c}}| & \leqslant\sum_{x\in\Omega^{c}}|e^{-i\langle x,z\rangle}(-\Delta)^{l}\left(h_{z}|_{\Omega}\right)^{\ast}(x)|\\
 & \leqslant\underset{x\in\delta\Omega\cup\ldots\cup\delta_{l}\Omega\,}{\sum}|\underset{y\in B(x,l)\cap\Omega}{\sum}a_{xy}^{l}h_{z}(y)|\\
 & \leqslant\underset{x\in\delta\Omega\cup\ldots\cup\delta_{l}\Omega\,}{\sum}\underset{y\in B(x,l)\cap\Omega}{\sum}|a_{xy}^{l}|\eqqcolon|\partial^{l}\Omega|.
\end{align*}
Then we can estimate the left-hand side of (\ref{eq:=00005CDelta^=00007Bl=00007D})
as
\begin{align*}
|\langle h_{z},(-1)^{l}\Delta_{\Omega}^{l,\mathcal{D}}\left(h_{z}|_{\Omega}\right)\rangle_{\Omega}| & \leqslant|\langle h_{z},(-\Delta)^{l}\left(h_{z}|_{\Omega}\right)^{\ast}\rangle_{\mathbb{Z}^{d}}|+|\langle h_{z},(-\Delta)^{l}\left(h_{z}|_{\Omega}\right)^{\ast}\rangle_{\Omega^{c}}|\\
 & \leqslant(\Phi(z))^{l}|\Omega|+|\partial^{l}\Omega|,
\end{align*}
Similar to (\ref{eq:a_xy}), we define 
\[
(D+A)^{l}f(x)=\sum_{m=0}^{l}\left(\begin{array}{c}
l\\
m
\end{array}\right)D^{l-m}A^{m}f(x)\eqqcolon\underset{y\in B(x,l)}{\sum}b_{xy}^{l}f(y),
\]
where $D$ is the degree matrix and $A$ is the adjacency matrix.
Then 
\[
\underset{y\in B(x,l)}{\sum}|a_{xy}^{l}|\leqslant\underset{y\in B(x,l)}{\sum}b_{xy}^{l}\leqslant\underset{y}{\sum}\left((D+A)^{l}\right)_{xy}\leqslant4^{l}d^{l},
\]
Hence we can get a upper estimate 
\[
|\partial^{l}\Omega|=\underset{x\in\delta\Omega\cup\ldots\cup\delta_{l}\Omega\,}{\sum}\underset{y\in B(x,l)\cap\Omega}{\sum}|a_{xy}^{l}|\leqslant4^{l}d^{l}\left(|\delta\Omega|+\ldots|\delta_{l}\Omega|\right).
\]
In particular, for $l=1$ we have
\begin{align}
|\langle h_{z},-\Delta_{\Omega}^{\mathcal{D}}\left(h_{z}|_{\Omega}\right)\rangle_{\Omega}| & =\sum_{(x,y)\in E(\Omega)}|\nabla_{xy}h_{z}|^{2}+\sum_{(x,y)\in E(\Omega,\delta\Omega)}|h_{z}(x)|^{2}\label{eq: |=00005Cpartial^=00007B1=00007D=00005COmega|}\\
 & \leqslant\Phi(z)|\Omega|+|E(\Omega,\delta\Omega)|.\nonumber 
\end{align}
And for $l=2$ we have
\begin{align*}
\sum_{x\in\Omega^{c}}|\Delta^{2}\left(h_{z}|_{\Omega}\right)^{\ast}(x)| & \leqslant\sum_{\substack{x\in\delta\Omega\\
\\
}
}|\Delta^{2}\left(h_{z}|_{\Omega}\right)^{\ast}(x)|+\sum_{\substack{x\in\delta_{2}\Omega\\
\\
}
}|\Delta^{2}\left(h_{z}|_{\Omega}\right)^{\ast}(x)|\\
 & \eqqcolon\text{\mbox{I}}+\text{\mbox{II}}.
\end{align*}
For convenience, denote $E_{1}\coloneqq|E(\delta\Omega,\Omega)|,E_{2}\coloneqq|E(\delta\Omega)|,E_{3}\coloneqq|E(\delta\Omega,\delta_{2}\Omega)|$.
Then
\begin{align*}
\text{\mbox{I}} & =\sum_{\substack{x\in\delta\Omega\\
\\
}
}|\sum_{\substack{y\sim x\\
\\
}
}\left(\Delta\left(h_{z}|_{\Omega}\right)^{\ast}(y)-\Delta\left(h_{z}|_{\Omega}\right)^{\ast}(x)\right)|\\
 & \leqslant\sum_{\substack{(x,y)\in E(\delta\Omega,\Omega)\\
\\
}
}+\sum_{\substack{(x,y)\in E(\delta\Omega)\\
\\
}
}+\sum_{\substack{(x,y)\in E(\delta\Omega,\delta_{2}\Omega)\\
\\
}
}|\Delta\left(h_{z}|_{\Omega}\right)^{\ast}(y)-\Delta\left(h_{z}|_{\Omega}\right)^{\ast}(x)|\\
 & =\sum_{\substack{(x,y)\in E(\delta\Omega,\Omega)\\
\\
}
}\left|\Delta\left(h_{z}|_{\Omega}\right)^{\ast}(y)+\frac{\partial\left(h_{z}|_{\Omega}\right)^{\ast}}{\partial n}(x)\right|+\\
 & \sum_{\substack{(x,y)\in E(\delta\Omega)\\
\\
}
}\left|-\frac{\partial\left(h_{z}|_{\Omega}\right)^{\ast}}{\partial n}(y)+\frac{\partial\left(h_{z}|_{\Omega}\right)^{\ast}}{\partial n}(x)\right|+\sum_{\substack{(x,y)\in E(\delta\Omega,\delta_{2}\Omega)\\
\\
}
}\left|\frac{\partial\left(h_{z}|_{\Omega}\right)^{\ast}}{\partial n}(x)\right|\\
 & \leqslant4dE_{1}+(E_{1}+2E_{2}+E_{3})\underset{x\in\delta\Omega}{\text{max}}\left|\frac{\partial\left(h_{z}|_{\Omega}\right)^{\ast}}{\partial n}(x)\right|.
\end{align*}
\begin{align*}
\text{\mbox{II}} & \leqslant\sum_{\substack{x\in\delta_{2}\Omega\\
\\
}
}|\sum_{\substack{y\sim x\\
\\
}
}\left(\Delta\left(h_{z}|_{\Omega}\right)^{\ast}(y)-\Delta\left(h_{z}|_{\Omega}\right)^{\ast}(x)\right)|\\
 & \leqslant\sum_{\substack{(x,y)\in E(\delta_{2}\Omega,\delta\Omega)\\
\\
}
}|\Delta\left(h_{z}|_{\Omega}\right)^{\ast}(y)|\\
 & \leqslant\sum_{\substack{(x,y)\in E(\delta_{2}\Omega,\delta\Omega)\\
\\
}
}\sum_{\substack{(y,y')\in E(\delta\Omega,\Omega)\\
\\
}
}|h_{z}(y')|\leqslant E_{3}E_{1}.
\end{align*}
Hence
\begin{equation}
|\partial^{2}\Omega|\leqslant4dE_{1}+(E_{1}+2E_{2}+E_{3})\underset{x\in\delta\Omega}{\text{max}}\left|\frac{\partial\left(h_{z}|_{\Omega}\right)^{\ast}}{\partial n}(x)\right|+E_{3}E_{1}.\label{eq: |=00005Cpartial^=00007B2=00007D=00005COmega|}
\end{equation}
\end{proof}

\section{\label{sec:Proof-of-Theorem}Estimates of Dirichlet poly-Laplace
eigenvalues }

In this section, we follow the methods in \cite{Li1983,Bauer2022,Kroger1994}
to derive the upper and lower bounds of Dirichlet poly-Laplace eigenvalues
on a finite subgraph $\Omega$ of $\mathbb{Z}^{d}$. Recall that
\[
(-1)^{l}\Delta_{\Omega}^{l,\mathcal{D}}f=\lambda^{l}f,\;\text{in \ensuremath{\Omega}}.
\]
There are $|\Omega|$ positive real eigenvalues (with multiplicities)
\[
0<\lambda_{1}^{l}\leqslant\lambda_{2}^{l}\leqslant\ldots\leqslant\lambda_{|\Omega|}^{l}.
\]
and denote the corresponding eigenfunctions by $\{\phi_{i}^{l}\}_{1\leqslant i\leqslant|\Omega|}$
respectively.

\subsection{Upper bound}

We begin with a general lemma (proved in \cite{Bauer2022}) on eigenvalues.
\begin{lem}
\label{lem: upper bound lemma}Let L be a self-adjoint, positive semidefinite
operator on a finite dimensional, Hermitian, complex vector space
$W$ with Hermitian inner product $\langle,\rangle$. Let $0\leqslant\gamma_{1}\leqslant\ldots\leqslant\gamma_{s}$
be its eigenvalues, and choose an orthonormal basis of eigenfunctions
$\{f_{i}\}_{1\leqslant i\leqslant|\Omega|}$, where $f_{i}$ corresponds
to $\gamma_{i}$. Then for any $1\leqslant k\leqslant s$ and any
vector $g\in W$ one has 
\begin{equation}
\begin{aligned}\gamma_{k+1}\langle g,g\rangle\leqslant\langle g,Lg\rangle+\sum_{j=1}^{k}(\gamma_{k+1}-\gamma_{j})|\langle g,f_{j}\rangle|^{2}.\end{aligned}
\label{eq: upper bound lemma of L}
\end{equation}
\end{lem}

By choosing $g=h_{z}|_{\Omega}$ and $f_{j}=\phi_{j}^{l}$, we can
derive an inequality between $\lambda^{l}{}_{k+1}$ and $\sum_{j=1}^{k}\lambda_{j}^{l}$
as follows.
\begin{lem}
\label{lem: upper bound lemma of poly-Laplace}Let $\Omega$ be a
finite subgraph of $\mathbb{Z}^{d}$. For any $1\leqslant k\leqslant|\Omega|$
and a measurable set $B\subset[-\pi,\pi]^{d}$, we have 
\[
\begin{aligned}\lambda_{k+1}^{l}(|\Omega||B|-(2\pi)^{d}k)\leqslant|\Omega|\cdot\int_{z\in B}(\Phi(z))^{l}dz-(2\pi)^{d}\sum_{j=1}^{k}\lambda_{j}^{l}+|B|\cdot|\partial^{l}\Omega|,\end{aligned}
\]
where $\Phi(z)=\sum_{i=1}^{d}(2-2\cos z_{i})$, and $\lambda_{k+1}^{l}=0$
if $k=|\Omega|$. 
\end{lem}

\begin{proof}
In Lemma~\ref{lem: upper bound lemma}, we choose 
\[
W=\mathbb{C}^{\Omega},\quad L=(-1)^{l}\Delta_{\Omega}^{l,\mathcal{D}},\quad g=h_{z}|_{\Omega},\quad\gamma_{j}=\lambda_{j}^{l},\quad f_{j}=\phi_{j}^{l},
\]
which satisfying the conditions of Lemma~\ref{lem: upper bound lemma}
since $(-1)^{l}\Delta_{\Omega}^{l,\mathcal{D}}$ is self-adjoint and
positive definite. Integrating both sides of the inequality (\ref{eq: upper bound lemma of L})
over a measurable set $B\subset[-\pi,\pi]^{d}$ yields 
\[
\begin{aligned}\lambda_{k+1}^{l}\int_{z\in B}\langle h_{z},h_{z}\rangle_{\Omega}\leqslant\int_{z\in B}\langle h_{z},(-1)^{l}\Delta_{\Omega}^{l,\mathcal{D}}(h_{z}|_{\Omega})\rangle_{\Omega}+\sum_{j=1}^{k}(\lambda_{k+1}^{l}-\lambda_{j}^{l})\int_{z\in B}|\langle h_{z},\phi_{j}^{l}\rangle_{\Omega}|^{2}.\end{aligned}
\]
By Lemma~\ref{lem: h_z}, the left-hand side is equal to $\lambda_{k+1}^{l}|B||\Omega|$,
and the right-hand side 
\begin{align*}
 & \leqslant\int_{z\in B}((\Phi(z))^{l}|\Omega|+|\partial^{l}\Omega|)dz+(2\pi)^{d}\sum_{j=1}^{k}(\lambda_{k+1}^{l}-\lambda_{j}^{l})\langle\phi_{j}^{l},\phi_{j}^{l}\rangle_{\Omega}\\
 & =|\Omega|\cdot\int_{z\in B}(\Phi(z))^{l}dz+|B|\cdot|\partial^{l}\Omega|+(2\pi)^{d}\sum_{j=1}^{k}(\lambda_{k+1}^{l}-\lambda_{j}^{l}),
\end{align*}
which completes the proof of this lemma. 
\end{proof}
Then we can prove the upper bound estimate (a) of Theorem \ref{thm: main thm}.
\begin{proof}[Proof of Theorem \ref{thm: main thm} (a)]
 In Lemma \ref{lem: upper bound lemma of poly-Laplace}, for $1\leqslant k\leqslant\min\{1,\frac{V_{d}}{2^{d}}\}|\Omega|$,
we choose the measurable set $B$ as a ball of radius $2\pi(\frac{k}{V_{d}|\Omega|})^{\frac{1}{d}}$
centered at the origin in $\mathbb{R}^{d}$, so that the radius $R=2\pi(\frac{k}{V_{d}|\Omega|})^{\frac{1}{d}}\leqslant\pi$
and $|B|=\frac{k(2\pi)^{d}}{|\Omega|}$. Thus we derive an inequality
\begin{equation}
\begin{aligned}(2\pi)^{d}\sum_{j=1}^{k}\lambda_{j}^{l}\leqslant|\Omega|\cdot\int_{z\in B}(\Phi(z))^{l}dz+|B|\cdot|\partial^{l}\Omega|.\end{aligned}
\label{eq: upper bound 1}
\end{equation}
In the rest of the proof, we need to estimate $\int_{z\in B}(\Phi(z))^{l}dz$.
And by calculation,
\[
\begin{aligned}(\Phi(z))^{l}=(\sum_{i=1}^{d}(2-2\cos z_{i}))^{l}=(\sum_{i=1}^{d}4\sin^{2}\frac{z_{i}}{2})^{l}\leqslant(\sum_{i=1}^{d}z_{i}^{2})^{l}=|z|^{2l}.\end{aligned}
\]
This yields that 
\[
\begin{aligned}\int_{z\in B}(\Phi(z))^{l}dz\leqslant\int_{z\in B}|z|^{2l}dz=dV_{d}\int_{0}^{R}r^{2l}\cdot r^{d-1}dr=\frac{dV_{d}}{d+2l}R^{2l+d}.\end{aligned}
\]
Applying the above inequality to (\ref{eq: upper bound 1}), we obtain
\[
\begin{aligned}\frac{1}{k}\sum_{j=1}^{k}\lambda_{j}^{l} & \leqslant\frac{dV_{d}}{d+2l}\frac{|\Omega|R^{d+2l}}{k(2\pi)^{d}}+\frac{|\partial^{l}\Omega|}{|\Omega|}\\
 & =(2\pi)^{2l}\frac{d}{d+2l}(\frac{k}{V_{d}|\Omega|})^{\frac{2l}{d}}+\frac{|\partial^{l}\Omega|}{|\Omega|}.
\end{aligned}
\]
Similarly, for $1\leqslant k\leqslant\min\{1,\frac{V_{d}}{2^{d+1}}\}|\Omega|$,
we choose $R'=2^{\frac{1}{d}}R$ , which implies $R'=2\pi(\frac{2k}{V_{d}|\Omega|})^{\frac{1}{d}}\leqslant\pi$
and $|B'|=\frac{2k(2\pi)^{d}}{|\Omega|}$. Then we repeat the above
proof and ignore the term $\sum_{j=1}^{k}\lambda_{j}^{l}$, obtaining
that 
\[
\begin{aligned}k(2\pi)^{d}\lambda_{k+1}^{l} & \leqslant|\Omega|\cdot\int_{z\in B'}(\Phi(z))^{l}dz+|B'|\cdot|\partial^{l}\Omega|\\
 & \leqslant|\Omega|\cdot\frac{dV_{d}}{d+2l}R'^{d+2l}+\frac{2k(2\pi)^{d}|\partial^{l}\Omega|}{|\Omega|},
\end{aligned}
\]
which implies that
\[
\lambda_{k+1}^{l}\leqslant(2\pi)^{2l}\frac{d\cdot2^{\frac{d+2l}{d}}}{d+2l}(\frac{k}{V_{d}|\Omega|})^{\frac{2l}{d}}+2\frac{|\partial^{l}\Omega|}{|\Omega|}.
\]
\end{proof}

\subsection{Lower bound}

To prove the lower bound estimate ($(b)$ of Theorem \ref{thm: main thm}),
we follow the method of Li and Yau \cite{Li1983}. 
\begin{lem}
\textup{(Modification of Lemma 1 from \cite{Li1983})} \label{lem: lower bound lemma}Let
$F$ be a real-valued function on $\mathbb{R}^{d}$ such that $0\leqslant F\leqslant M$
and 
\[
\int_{[-\pi,\pi]^{d}}F(z)dz\geqslant K.
\]
Assume
\[
(\frac{K}{MV_{d}})^{\frac{1}{d}}\leqslant\sqrt{6}<\pi.
\]
Then we have 
\[
\int_{[-\pi,\pi]^{d}}(\Phi(z))^{l}F(z)dz\geqslant\sum_{m=0}^{l}\left(\begin{array}{c}
l\\
m
\end{array}\right)\left(-\frac{1}{12}\right)^{m}\frac{dK}{d+2(l+m)}(\frac{K}{MV_{d}})^{\frac{2(l+m)}{d}}>0,
\]
where $\Phi(z)=\sum_{i=1}^{d}(2-2\cos z_{i})$. 
\end{lem}

\begin{proof}
Without loss of generality, assume that $\int_{[-\pi,\pi]^{d}}F(z)dz=K$
by decreasing $F$ as necessary. Then prove the lemma in two steps.

\textbf{Step 1:} First we are going to construct a radially symmetric
function $\varphi(z):[-\pi,\pi]^{d}\rightarrow\mathbb{R}$ such that
$0\leqslant\varphi(z)\leqslant(\Phi(z))^{l}$. Since $2-2\cos x\geq x^{2}-x^{4}/12$
for all $x$. Then 
\[
\sum_{i=1}^{d}(2-2\cos z_{i})\geqslant\sum_{i=1}^{d}(z_{i}^{2}-\frac{1}{12}z_{i}^{4})\geqslant\sum_{i=1}^{d}z_{i}^{2}-\frac{1}{12}\left(\sum_{i=1}^{d}z_{i}^{2}\right)^{2}=|z|^{2}-\frac{1}{12}|z|^{4},
\]
which means for $|z|\leqslant\pi<2\sqrt{3}$
\[
(\Phi(z))^{l}=(\sum_{i=1}^{d}(2-2\cos z_{i}))^{l}\geqslant\left(|z|^{2}-\frac{1}{12}|z|^{4}\right)^{l}\eqqcolon H(|z|).
\]
Note that $H(|z|)$ is monotone increasing on $[0,\sqrt{6}]$. Then
we define 
\[
\varphi(z)=\left\{ \begin{aligned} & H(|z|),\quad & \ |z|\leqslant\sqrt{6},\\
 & H(\sqrt{6})=3^{l}, & \ \sqrt{6}<|z|\leqslant\pi.
\end{aligned}
\right.
\]
Since $\Phi(z)$ is monotone increasing in $[-\pi,\pi]^{d}$ along
each half line starting at $0$, then we get that $(\Phi(z))^{l}\geqslant\varphi(z)$
for all $z\in[-\pi,\pi]^{d}$. 

\textbf{Step 2:} Since $F\geqslant0$, we can give a lower bound of
$\int_{[-\pi,\pi]^{d}}(\Phi(z))^{l}F(z)dz$ as
\begin{equation}
\begin{aligned}\int_{[-\pi,\pi]^{d}}(\Phi(z))^{l}F(z)dz\geqslant\int_{[-\pi,\pi]^{d}}\varphi(z)F(z)dz.\end{aligned}
\end{equation}
Let $R=(\frac{K}{MV_{d}})^{\frac{1}{d}}\leqslant\sqrt{6}<\pi$ such
that $MR^{d}V_{d}=K$, and define 
\[
\widetilde{F}(z)=\left\{ \begin{aligned} & M\  & ,\ |z|\leqslant R,\\
 & \ 0 & ,\ |z|>R.
\end{aligned}
\right.
\]
Then we have
\[
\int_{[-\pi,\pi]^{d}}\widetilde{F}(z)dz=\int_{[-\pi,\pi]^{d}}F(z)dz=K.
\]
We claim that $\widetilde{F}$ minimizes the integral $\int_{[-\pi,\pi]^{d}}\varphi(z)F(z)dz$
for all functions $F$ satisfying $0\leqslant F\leqslant M$ and $\int_{[-\pi,\pi]^{d}}F(z)dz=K$.
Since $\varphi(z)$ is monotonic and radially symmetric, then
\[
\begin{aligned} & \int_{[-\pi,\pi]^{d}}\varphi(z)(F(z)-\widetilde{F}(z))dz\\
= & \int_{[-\pi,\pi]^{d}\setminus B_{R}}\varphi(z)F(z)dz-\int_{B_{R}}\varphi(z)(M-F(z))dz\\
\geqslant & \varphi(R)\int_{[-\pi,\pi]^{d}\setminus B_{R}}F(z)dz-\varphi(R)\int_{B_{R}}(M-F(z))dz\\
= & \varphi(R)(\int_{[-\pi,\pi]^{d}}F(z)dz-\int_{B_{R}}Mdz)=0.
\end{aligned}
\]
Thus we have the estimate 
\[
\begin{aligned} & \int_{[-\pi,\pi]^{d}}(\Phi(z))^{l}F(z)dz\geqslant\int_{[-\pi,\pi]^{d}}\varphi(z)F(z)dz\geqslant\int_{[-\pi,\pi]^{d}}\varphi(z)\widetilde{F}(z)dz\\
= & M\int_{B_{R}}\varphi(z)dz=MdV_{d}\int_{0}^{R}H(r)r^{d-1}dr\\
= & MdV_{d}\left(\sum_{m=0}^{l}\left(\begin{array}{c}
l\\
m
\end{array}\right)\left(-\frac{1}{12}\right)^{m}\frac{R^{d+2(l+m)}}{d+2(l+m)}\right)\\
= & \sum_{m=0}^{l}\left(\begin{array}{c}
l\\
m
\end{array}\right)\left(-\frac{1}{12}\right)^{m}\frac{dK}{d+2(l+m)}(\frac{K}{MV_{d}})^{\frac{2(l+m)}{d}},
\end{aligned}
\]
which is positive when $R=(\frac{K}{MV_{d}})^{\frac{1}{d}}\leqslant\sqrt{6}$.
\end{proof}
Then we can prove the lower bound estimate $(b)$ of Theorem \ref{thm: main thm}.
\begin{proof}[Proof of Theorem \ref{thm: main thm} (b)]
 Suppose eigenfunctions $\{\phi_{j}^{l}\}_{j=1}^{k}$ form a standardized
orthogonal basis in $\ell^{2}(\Omega)$. We define $P_{j}$ as the
projection operator on the space spanned by $\phi_{j}^{l}$, and $P$
as the projection operator on the space spanned by $\{\phi_{j}^{l}\}_{j=1}^{k}$.
Let $h_{z}(x)=e^{i\langle z,x\rangle}$ as before. To apply Lemma~\ref{lem: lower bound lemma},
let 
\[
F_{j}(z)=\langle P_{j}h_{z},P_{j}h_{z}\rangle_{\Omega}=\|P_{j}h_{z}\|_{\ell^{2}(\Omega)}^{2},
\]
\[
F(z)=\sum_{j=1}^{k}F_{j}(z)=\|Ph_{z}\|_{\ell^{2}(\Omega)}^{2}.
\]
By Lemma~\ref{lem: h_z}, we have 
\[
F(z)=\|Ph_{z}\|_{\ell^{2}(\Omega)}^{2}=\sum_{j=1}^{k}|\langle\phi_{j}^{l},h_{z}\rangle_{\Omega}|^{2}\leqslant\|h_{z}\|_{\ell^{2}(\Omega)}^{2}=|\Omega|,
\]
\[
\int_{[-\pi,\pi]^{d}}F(z)dz=\sum_{j=1}^{k}\int_{[-\pi,\pi]^{d}}|\langle\phi_{j}^{l},h_{z}\rangle_{\Omega}|^{2}dz=(2\pi)^{d}\sum_{j=1}^{k}\langle\phi_{j}^{l},\phi_{j}^{l}\rangle_{\Omega}=k(2\pi)^{d},
\]
\[
\begin{aligned}\int_{[-\pi,\pi]^{d}}(\Phi(z))^{l}F(z)dz & =\sum_{j=1}^{k}\int_{[-\pi,\pi]^{d}}(\Phi(z))^{l}|\langle\phi_{j}^{l},h_{z}\rangle_{\Omega}|^{2}dz\\
 & =(2\pi)^{d}\sum_{j=1}^{k}\langle\phi_{j}^{l},(-1)^{l}\Delta_{\Omega}^{l,\mathcal{D}}\phi_{j}^{l}\rangle_{\Omega}=(2\pi)^{d}\sum_{j=1}^{k}\lambda_{j}^{l}.
\end{aligned}
\]
By Lemma~\ref{lem: lower bound lemma} with $M=|\Omega|,K=k(2\pi)^{d}$,
we get 
\begin{align*}
\lambda_{k}^{l}(\Omega) & \geqslant\frac{1}{k}\sum_{j=1}^{k}\lambda_{j}^{l}(\Omega)\\
 & \geqslant\sum_{m=0}^{l}\left(\begin{array}{c}
l\\
m
\end{array}\right)\left(-\frac{1}{12}\right)^{m}(2\pi)^{2(l+m)}\frac{d}{d+2(l+m)}(\frac{k}{|\Omega|V_{d}})^{\frac{2(l+m)}{d}},
\end{align*}
which is positive when $k\leqslant\min\{1,(\frac{\sqrt{6}}{2\pi})^{d}V_{d}\}|\Omega|$.
\end{proof}

\subsection{Comparison between eigenvalues of l and 2l-order poly-Laplace}

Next we want to compare the eigenvalue of $\Delta_{\Omega}^{2l,\mathcal{D}}$
and the square of eigenvalue of $(-1)^{l}\Delta_{\Omega}^{l,\mathcal{D}}$.
First we prove a lemma as follows.
\begin{lem}
\label{lem:bi-laplace and laplace}Let $\Omega$ be a finite subgraph
of $G$ and $l\in\mathbb{N}_{+}$, then for all $f\in C(\Omega)$,
\[
\langle\left(\Delta_{\Omega}^{l,\mathcal{D}}\right)^{2}f,f\rangle_{\Omega}\leqslant\langle\Delta_{\Omega}^{2l,\mathcal{D}}f,f\rangle_{\Omega}.
\]
\end{lem}

\begin{proof}
For all $f\in C(\Omega)$, we have 
\begin{align*}
\langle\left(\Delta_{\Omega}^{l,\mathcal{D}}\right)^{2}f,f\rangle_{\Omega} & =\langle\Delta_{\Omega}^{l,\mathcal{D}}f,\Delta_{\Omega}^{l,\mathcal{D}}f\rangle_{\Omega}=\langle\Delta^{l}f^{\ast}|_{\Omega},\Delta^{l}f^{\ast}|_{\Omega}\rangle_{\Omega}\\
 & \leqslant\langle\Delta^{l}f^{\ast},\Delta^{l}f^{\ast}\rangle_{G}=\langle\Delta^{2l}f^{\ast},f^{\ast}\rangle_{G}\\
 & =\langle\Delta^{2l}f^{\ast}|_{\Omega},f\rangle_{\Omega}=\langle\Delta_{\Omega}^{2l,\mathcal{D}}f,f\rangle_{\Omega}.
\end{align*}
\end{proof}
The following lemma states that there is no $\ell^{2}$-eigenfunction
of poly-Laplace on $\mathbb{Z}^{d}$. 
\begin{lem}
\label{lem:eigenfunction on Z^d}For $l\in\mathbb{N}_{+}$ and $\lambda^{l}>0$,
if $f\in\ell^{2}(\mathbb{Z}^{d})$ satisfies
\begin{equation}
(-\Delta)^{l}f(x)=\lambda^{l}f(x)\;\text{on \ensuremath{\mathbb{Z}^{d}}}.\label{eq:eigenfunction on Z^d}
\end{equation}
Then $f\equiv0$.
\end{lem}

\begin{proof}
Since $\ell^{2}(\mathbb{Z}^{d})$ is the completion of $C_{0}(\mathbb{Z}^{d})$
in $\ell^{2}$ norm. Assume that there exists a function $f\in C_{0}(\mathbb{Z}^{d})$
satisfying (\ref{eq:eigenfunction on Z^d}). By the Fourier transform
of both sides of (\ref{eq:eigenfunction on Z^d}), we have 
\[
(\Phi(z))^{l}\hat{f}(z)=\lambda^{l}\hat{f}(z)\text{ on }[-\pi,\pi]^{d}.
\]
Then 
\[
\left\{ z\in[-\pi,\pi]^{d}|\,\hat{f}(z)\neq0\right\} \subseteq\left\{ z\in[-\pi,\pi]^{d}|\,(\Phi(z))^{l}=\lambda^{l}\right\} ,
\]
which is a set of Lebesgue measure zero. Hence $\hat{f}(z)\overset{a.e.}{=}0$,
and then $f(x)=0$ on $\mathbb{Z}^{d}$.
\end{proof}
Then we are ready to prove that the eigenvalues of $\Delta_{\Omega}^{2l,\mathcal{D}}$
are at least as large as the squares of the eigenvalues of $(-1)^{l}\Delta_{\Omega}^{l,\mathcal{D}}$. 
\begin{proof}[Proof of Theorem \ref{thm:bi-laplace and laplace}]
By the Min-Max-Theorem and Lemma \ref{lem:bi-laplace and laplace},
we know 
\begin{align*}
\left(\lambda_{k}^{l}\right)^{2} & =\underset{\begin{array}{c}
M\subseteq C(\Omega)\\
\text{dim}M=k
\end{array}}{\text{min}}\underset{\begin{array}{c}
f\in M\\
\|f\|_{\ell^{2}}=1
\end{array}}{\text{max}}\langle(-1)^{l}\Delta_{\Omega}^{l,\mathcal{D}}f,f\rangle_{\Omega}^{2}\\
 & \leqslant\underset{\begin{array}{c}
M\subseteq C(\Omega)\\
\text{dim}M=k
\end{array}}{\text{min}}\underset{\begin{array}{c}
f\in M\\
\|f\|_{\ell^{2}}=1
\end{array}}{\text{max}}\langle\Delta_{\Omega}^{l,\mathcal{D}}f,\Delta_{\Omega}^{l,\mathcal{D}}f\rangle_{\Omega}\langle f,f\rangle_{\Omega}\\
 & =\underset{\begin{array}{c}
M\subseteq C(\Omega)\\
\text{dim}M=k
\end{array}}{\text{min}}\underset{\begin{array}{c}
f\in M\\
\|f\|_{\ell^{2}}=1
\end{array}}{\text{max}}\langle\left(\Delta_{\Omega}^{l,\mathcal{D}}\right)^{2}f,f\rangle_{\Omega}\\
 & \leqslant\underset{\begin{array}{c}
M\subseteq C(\Omega)\\
\text{dim}M=k
\end{array}}{\text{min}}\underset{\begin{array}{c}
f\in M\\
\|f\|_{\ell^{2}}=1
\end{array}}{\text{max}}\langle\Delta_{\Omega}^{2l,\mathcal{D}}f,f\rangle_{\Omega}=\lambda_{k}^{2l},
\end{align*}
Moreover, if $\Omega$ is a finite subgraph of $\mathbb{Z}^{d}$,
and the equality $\left(\lambda_{k}^{l}\right)^{2}=\lambda_{k}^{2l}$
holds, then 
\[
\lambda_{k}^{2l}=\underset{\begin{array}{c}
f\in M_{0}\\
\|f\|_{\ell^{2}}=1
\end{array}}{\text{max}}\langle\Delta_{\Omega}^{2l,\mathcal{D}}f,f\rangle_{\Omega}=\underset{\begin{array}{c}
f\in M_{0}\\
\|f\|_{\ell^{2}}=1
\end{array}}{\text{max}}\langle(-1)^{l}\Delta_{\Omega}^{l,\mathcal{D}}f,f\rangle_{\Omega}^{2}=\left(\lambda_{k}^{l}\right)^{2}.
\]
Suppose $f_{0}\in M_{0}$ such that $\|f_{0}\|_{\ell^{2}}=1$ and
\[
\left(\lambda_{k}^{l}\right)^{2}=\langle(-1)^{l}\Delta_{\Omega}^{l,\mathcal{D}}f_{0},f_{0}\rangle_{\Omega}^{2}.
\]
Then
\begin{align*}
\left(\lambda_{k}^{l}\right)^{2} & =\langle(-\Delta)^{l}f_{0}^{\ast},f_{0}^{\ast}\rangle_{\mathbb{Z}^{d}}^{2}\leqslant\langle(-\Delta)^{l}f_{0}^{\ast},(-\Delta)^{l}f_{0}^{\ast}\rangle_{\mathbb{Z}^{d}}\langle f_{0}^{\ast},f_{0}^{\ast}\rangle_{\mathbb{Z}^{d}}\\
 & =\langle(-\Delta)^{2l}f_{0}^{\ast},f_{0}^{\ast}\rangle_{\mathbb{Z}^{d}}=\langle\Delta_{\Omega}^{2l,\mathcal{D}}f_{0},f_{0}\rangle_{\Omega}\leqslant\lambda_{k}^{2l},
\end{align*}
where the first inequality follows from the Cauchy-Schwarz inequality.
Since both sides of the above formula are equal, then there is a positive
constant $\mu$ such that $(-\Delta)^{l}f_{0}^{\ast}=\mu f_{0}^{\ast}$.
By Lemma \ref{lem:eigenfunction on Z^d} we have $f_{0}^{\ast}\equiv0,$
which contradicts to $\|f_{0}\|_{\ell^{2}}=1$.
\end{proof}
We also have the same estimate in the continuous settings. 
\begin{thm}
\label{thm: continuous laplace and bi-laplace}Let $\Omega$ be a
connected bounded domain with smooth boundary in $\mathbb{R}^{d}$
and $\nu$ be the outward unit normal vector field of $\partial\Omega$.
For a positive integer $l$, if $\lambda_{k}^{il}$ ($i=1,2$) are
the $k$-th eigenvalues of the Dirichlet poly-Laplace problems
\[
\begin{cases}
\begin{array}{c}
(-\Delta)^{il}f=\lambda^{il}f,\\
f=\frac{\partial f}{\partial\nu}=\ldots=\frac{\partial^{il-1}f}{\partial^{il-1}\nu}=0,
\end{array} & \begin{array}{c}
\text{in }\Omega,\\
\text{on }\partial\Omega,
\end{array}\end{cases}
\]
respectively. Then 
\[
\left(\lambda_{k}^{l}\right)^{2}\leqslant\lambda_{k}^{2l}.
\]
\end{thm}

\begin{proof}
Recall
\[
H_{0}^{il}(\Omega)=\left\{ f:\|f\|_{H^{il}}=\underset{|\alpha|\leqslant il}{\sum}\|D^{\alpha}f\|_{L^{2}}<+\infty,\,D^{\alpha}f|_{\partial\Omega}=0\text{ for }|\alpha|\leqslant il-1\right\} ,
\]
and note that $H_{0}^{il}(\Omega)$ is the completion of $C_{0}^{\infty}(\Omega)$
in $H^{il}$ norm. Then
\begin{align*}
\left(\lambda_{k}^{l}\right)^{2} & =\underset{\begin{array}{c}
M\subseteq H_{0}^{l}(\Omega)\\
\text{dim}M=k
\end{array}}{\text{inf}}\underset{\begin{array}{c}
f\in M\\
\|f\|_{L^{2}}=1
\end{array}}{\text{sup}}\langle\nabla^{l}f,\nabla^{l}f\rangle_{\Omega}^{2}\\
 & =\underset{\begin{array}{c}
M\subseteq C_{0}^{\infty}(\Omega)\\
\text{dim}M=k
\end{array}}{\text{inf}}\underset{\begin{array}{c}
f\in M\\
\|f\|_{L^{2}}=1
\end{array}}{\text{sup}}\langle(-\Delta)^{l}f,f\rangle_{\Omega}^{2}\\
 & \leqslant\underset{\begin{array}{c}
M\subseteq C_{0}^{\infty}(\Omega)\\
\text{dim}M=k
\end{array}}{\text{inf}}\underset{\begin{array}{c}
f\in M\\
\|f\|_{L^{2}}=1
\end{array}}{\text{sup}}\langle(-\Delta)^{l}f,(-\Delta)^{l}f\rangle_{\Omega}\langle f,f\rangle_{\Omega}\\
 & =\underset{\begin{array}{c}
M\subseteq H_{0}^{2l}(\Omega)\\
\text{dim}M=k
\end{array}}{\text{inf}}\underset{\begin{array}{c}
f\in M\\
\|f\|_{L^{2}}=1
\end{array}}{\text{sup}}\langle(-\Delta)^{l}f,(-\Delta)^{l}f\rangle_{\Omega}=\lambda_{k}^{2l}.
\end{align*}
\end{proof}

\section{Appendix\label{sec:appendix}}

The following numerical experiment illustrates that on the path graph
$[0,n]\subseteq\mathbb{Z}$, the ratio of the squares of discrete
Dirichlet Laplace eigenvalues and Dirichlet bi-Laplace eigenvalues
in Remark \ref{rem: bilaplace estimate} (3) approximates the continuous
case $(0,n)\subseteq\mathbb{R}$, that is,
\[
\frac{\left(\lambda_{k}^{1}([0,n])\right)^{2}}{\lambda_{k}^{2}([0,n])}\to c_{k}=\frac{\left(\lambda_{k}^{1}((0,1))\right)^{2}}{\lambda_{k}^{2}((0,1))}<1\text{ as }n\to+\infty,
\]
suggesting that the definition (\ref{eq: main poly-laplace Dirichlet eigenvalue problem})
is reasonable. 
\begin{center}
\includegraphics[scale=0.45]{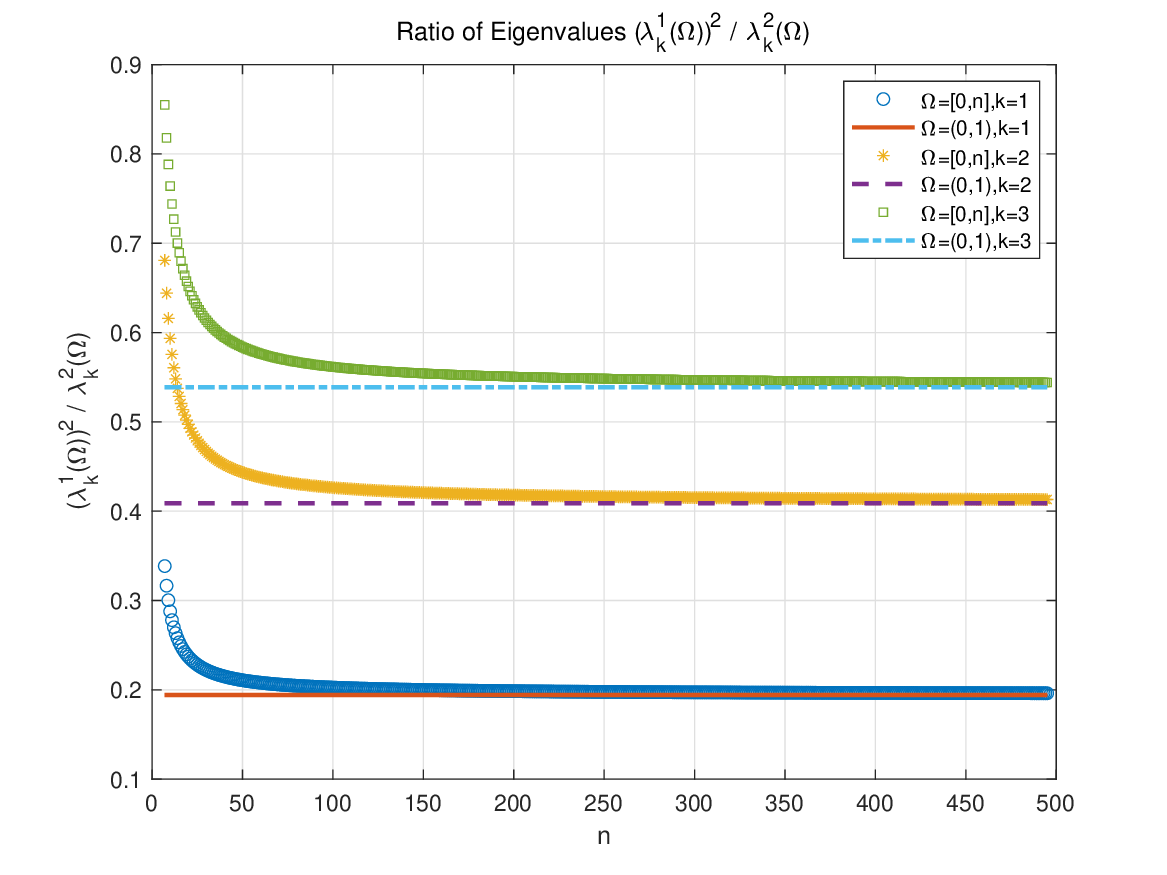}
\par\end{center}

\begin{center}
\textbf{Figure 1.} This figure shows the convergence of ratio of eigenvalues.
\par\end{center}

\textbf{Acknowledgements.} The authors would like to thank Florentin
Münch for helpful discussions. B. Hua is supported by NSFC, No. 12371056,
and by Shanghai Science and Technology Program {[}Project No. 22JC1400100{]}.

\textbf{Conflicts of Interests.} The authors declared no potential
conflicts of interests with respect to this article.

\textbf{Ethics Approval.} This study did not involve any human participants
or animals, and therefore, ethical approval was not required.

\textbf{Data Availability.} Data availability is not applicable to
this article as no new data were created or analyzed in this study. 

\bibliographystyle{plain}
\bibliography{../poly-laplace_ref}

\end{document}